\documentclass[12pt]{amsart}
\usepackage[cp1251]{inputenc}
\usepackage[russian]{babel}
\usepackage{amssymb}
\def\No{No}

\theoremstyle{plain}

\newtheorem{lemma}{Лемма}

\theoremstyle{definition}
\newtheorem{definition}{Определение}

\newtheorem{remark}{Замечание}

\theoremstyle{plain}
\newtoks\thehProclaim
\newtheorem*{Proclaim}{\the\thehProclaim}
\newenvironment{proclaim}[1]{\thehProclaim{#1}\begin{Proclaim}}{\end{Proclaim}}

\theoremstyle{definition}
\newtoks{\thehRemark}
\newtheorem*{Remark}{\the\thehRemark}

\begin{document}

\title[Multipliers in Bessel potential spaces]{Multipliers in Bessel potential spaces.  The case of different sign smooth indices}

\author{Belyaev A.A.}

\address{Department of Mechanics and Mathematics\\
Lomonosov Moscow State University\\
119991, Leninskie Gori, 1, Main MSU Building, 
Russia}

\email{alexei.a.belyaev@gmail.com}

\author{Shkalikov A.A.}

\address{Department of Mechanics and Mathematics\\
Lomonosov Moscow State University\\
119991, Leninskie Gori, 1, Main MSU Building
Russia}

\email{shkalikov@mi.ras.ru}



\begin{abstract}
The objective of this paper is to describe the space of multipliers acting from a Bessel potential space  $H^s_p(\mathbb R^n)$ into another space  $H^{-t}_q(\mathbb R^n)$, provided that  the smooth indices of these spaces have different signs, i.e. $s, t \geqslant 0$. This space of multipliers consists
of distributions  $u$, such that for all $\varphi \in H^s_p(\mathbb R^n)$ the product $\varphi \cdot u$ is well-defined and belongs to the space  $H^{-t}_q(\mathbb R^n)$. We succeed to describe this space explicitly,  provided that  $p \leqslant q$  and one of the following conditions
$$
s \geqslant t \geqslant 0, \ s > n/p \ \ \, \text{or} \ \ \, t \geqslant s \geqslant 0, \  t > n/q' \quad (\: \text{где} \; 1/q +1/q' = 1),
$$
holds. In this case one has
$$
M[H^s_p(\mathbb{R}^n) \to H^{-t}_{q}(\mathbb{R}^n)] = H^{-t}_{q, \: unif}(\mathbb{R}^n) \cap H^{-s}_{p', \: unif}(\mathbb{R}^n),
$$
where $H^\gamma_{r, \: unif}(\mathbb{R}^n), \: \gamma \in \mathbb{R}, \: r > 1$ is the scale of uniformly localized Bessel potential spaces.
%этого равенства.

In particular but important case  $s = t < n/\max (p,q')$ we prove two-sided continuous embeddings
$$
H^{-s}_{r_1, \: unif}(\mathbb{R}^n) \subset M[H^s_p(\mathbb{R}^n) \to H^{-s}_q(\mathbb{R}^n)] \subset H^{-s}_{r_2, \: unif}(\mathbb{R}^n),
$$
where  $r_2 = \max (p', q), \  r_1 =[s/n-(1/p -1/q)]^{-1}$.
\end{abstract}

\thanks{
The results presented in Theorem 1 of this work received support from the Russian Science Fundation, RSCF, grant $\No$ 17-11-01215; the results presented in Theorem 2 received support from the Russian Foundatioon of Fundamental Research, RFFR, grant $\No$ 16-01-00706.}

%\translator{П. П. Петров}

\date{5 January 2018}

\maketitle

\begin{center}
\section{ВВЕДЕНИЕ}
\end{center}
\medskip

Цель работы --- изучить мультипликаторы, действующие из одного пространства бесселевых потенциалов $H^s_p(\mathbb R^n)$ в другое пространство $H^{-t}_q(\mathbb R^n)$, где $s, t \geqslant 0, \: p, q \in (1, +\infty)$. Пространство всех таких мультипликаторов будем обозначать через $M[H^s_p(\mathbb{R}^n) \to H^{-t}_q(\mathbb{R}^n)]$. Основное внимание уделено случаю, когда $p \leqslant q$ и выполнено одно из условий
\begin{equation}\label{1}
s \geqslant t \geqslant 0, \ s > n/p \ \ \, \text{или} \ \ \, t \geqslant s \geqslant 0, \ t > n/q',
\end{equation}
где число $q'$ определяется как сопряжённое по Гёльдеру к числу $q$. Мы покажем, что в этом случае пространство мультипликаторов допускает явное описание в терминах шкалы пространств $H^\gamma_{r, \: unif}(\mathbb{R}^n)$ равномерно локализованных бесселевых потенциалов. Более того, оказывется, что такое явное описание возможно только при $p\leqslant q$ и выполнении одного из условий \eqref{1}.

Шкала пространств $H^\gamma_{r, \: unif}(\mathbb{R}^n)$ для положительных индексов гладкости $\gamma$ была введена Р.~Стрихартцем в работе \cite{Str1}. В этой работе был также получен первый важный результат об описании пространств мультипликаторов в терминах этой шкалы, %$H^\gamma_{r, \: unif}(\mathbb{R}^n)$,
состоящий в том, что
$$
M[H^s_p(\mathbb{R}^n) \to H^s_p(\mathbb{R}^n)] = H^s_{p, \: unif}(\mathbb{R}^n) \quad \text{при} \ \, p > 1, \: s > \frac np \, .
$$
В дальнейшем задача по изучению мультипликаторов не только в пространствах бесселевых потенциалов, но и в других функциональных пространствах соболевского типа, таких, как пространства Лизоркина--Трибе- ля, Бесова и др., исследовалась многими авторами. Отметим, в частности, монографии В.\,Г.~Мазьи и Т.\,О.~Шапошниковой \cite{MShbook}, В.~Зикеля и Т.~Рунста \cite{RSbook}, а также работы Ж.~Бурдо \cite{Bour}, Й.~Франке \cite{Fr}, В.~Зикеля, Х.~Трибеля и И.~Смирнова \cite{SiTr, Si, SS}.

В случае $s > n/p$, который называют {\it стрихартцевским}, наиболее общий результат об описании мультипликаторов в пространствах бесселевых потенциалов
$M[H^s_p(\mathbb{R}^n) \to H^t_q(\mathbb{R}^n)]$ с индексами гладкости одинакового знака получен в недавней работе авторов \cite{BelShk}. В случае же $s \leqslant n/p$ явное описание пространств мультипликаторов получить не удаётся, но при этом плодотворным оказывается %развитый в цикле работ В.\,Г.~Мазьи, Т.\,О.~Шапошниковой и И.\,Э.~Вербицкого и изложенный в монографии \cite{MShbook}
подход, основанный на характеризации пространств мультипликаторов в терминах ёмкостей, детальное изложение которого содержится в \cite{MShbook}.

Начало систематическому изучению пространств мультипликаторов в случае, когда индексы гладкости пространств $H^s_p(\mathbb{R}^n)$ и $H^t_q(\mathbb{R}^n)$ разного знака, было положено М.\,И.~Нейман-Заде и А.\,А.~Шкаликовым в работе \cite{NZSh1}, где была отмечена важная роль теории мультипликаторов в приложениях к спектральной теории дифференциальных операторов, в частности, при изучении оператора Шрёдингера и более общих эллиптических дифференциальных операторов с сингулярными потенциалами. Предложенный в \cite{NZSh1} подход к изучению мультипликаторов в пространствах бесселевых потенциалов был развит в работах \cite{BSh, NZSh2}, где были получены следующие результаты:

\begin{align*}
1) \: & M[H^s_p(\mathbb{R}^n) \to H^{-s}_{p'}(\mathbb{R}^n)] = H^{-s}_{p', \: unif}(\mathbb{R}^n) \; \; \mbox{при} \; \; 1 < p \leqslant 2, \: s > \frac{n}{p} \: ;\\
2) \: & M[H^s_p(\mathbb{R}^n) \to H^{-s}_p(\mathbb{R}^n)] = H^{-s}_{\max (p, p'), \: unif}(\mathbb{R}^n) \; \; \mbox{при} \; \; p > 1, \: s > \frac{n}{\max (p, p')} \: ;\\
3) \: & M[H^s_2(\mathbb{R}^n) \to H^{-t}_2(\mathbb{R}^n)] = H^{-t}_{2, \: unif}(\mathbb{R}^n) \; \; \mbox{при} \; \; s \geqslant t \geqslant 0, \: s > \frac{n}{2} \: .
\end{align*}

В случае, когда показатели гладкости являются целыми и соответствующие пространства бесселевых потенциалов совпадают с классическими пространствами Соболева $W^s_p(\mathbb R^n)$, в работе В.\,Г.~Мазьи и Т.\,О.~Шапошниковой \cite{MSh} с помощью других методов было установлено, что
$$
4) \: M[W^k_p(\mathbb{R}^n) \to W^{-l}_p(\mathbb{R}^n)] = W^{-l}_{p, \: unif}(\mathbb{R}^n) \cap W^{-k}_{p', \: unif}(\mathbb{R}^n),
$$
если $k, \: l \in \mathbb{N}$ и выполнено одно из условий
$$
a) \: k \geqslant l, \; k > \frac{n}{p} \: ; \; \; b) \: l \geqslant k, \; l > \frac{n}{p'} \: .
$$

Отметим также работы В.\,Г.~Мазьи и И.\,Э.~Вербицкого \cite{MV1}, \cite{MV2}, С.~Гала и П.\,Ж.~Лемари-Рьёссе \cite{LRG} и П.~Жермена \cite{GerPaper}, где мультипликаторы, действующие между пространствами бесселевых потенциалов с индексами гладкости разного знака, изучались с помощью отличающихся от развитых в цикле статей \cite{BSh, NZSh1, NZSh2} методов. Шкала пространств равномерно локализованных бесселевых потенциалов в этих работах не использовалась. Из недавних работ упомянем статью Л.\,К.~Кусаиновой, А.\,Х.~Мырзагалиевой и Я.\,Т.~Султанаева \cite{KMS}, где действующие в весовых пространствах Соболева мультипликаторы применялись для изучения оператора Шрёдингера с сингулярным потенциалом.

Методы работ \cite{BSh, NZSh2} были развиты А.\,А.~Беляевым в статье \cite{Bel2}, в которой были получены следующие результаты:

\medskip

%\begin{align*}
5) если $1 < p \leqslant q$ и $s > n/\max(p, q')$, то
$$
M[H^s_p(\mathbb{R}^n) \to H^{-s}_{q}(\mathbb{R}^n)] = H^{-s}_{\max(p', q), \: unif}(\mathbb{R}^n);
$$

6) если $1 < p \leqslant 2$ и $\max(s, t) > n/p$, то
$$
M[H^s_p(\mathbb{R}^n) \to H^{-t}_{p'}(\mathbb{R}^n)] = H^{-\min (s, t)}_{p', \: unif}(\mathbb{R}^n).
$$
%\end{align*}

Сформулированные здесь утверждения 1)---6) представляют собой частные случаи следующей теоремы, являющейся главным результатом настоящей работы.

\begin{proclaim}{Теорема 1} Пусть $p, \: q > 1$ и $s, \: t \geqslant 0$. Тогда справедливо непрерывное вложение
\begin{equation}\label{emb}
M[H^s_p(\mathbb{R}^n) \to H^{-t}_{q}(\mathbb{R}^n)] \subset H^{-t}_{q, \: unif}(\mathbb{R}^n) \cap H^{-s}_{p', \: unif}(\mathbb{R}^n).
\end{equation}
Если, к тому же, $p \leqslant q$ и выполняется одно из двух условий
\begin{equation}\label{2}
a) \; s \geqslant t, \; s > \frac{n}{p} \: ; \; \quad b) \; t \geqslant s, \; t > \frac{n}{q'} \: ,
\end{equation}
то имеет место совпадение пространств
\begin{equation}\label{equi}
M[H^s_p(\mathbb{R}^n) \to H^{-t}_{q}(\mathbb{R}^n)] = H^{-t}_{q, \: unif}(\mathbb{R}^n) \cap H^{-s}_{p', \: unif}(\mathbb{R}^n)
\end{equation}
и нормы этих пространств эквивалентны.
 %Для справедливости этого равенства выполнение условия  $1< p \leqslant q$ и одного из условий \eqref{2} является %не только достаточным, но и необходимым.
\end{proclaim}

Основная новизна этого результата состоит в том, что мы получаем описание пространства мультипликаторов для случая, когда все четыре индекса $s, \: t, \: p, \: q$ могут быть различными. В этом случае доказательство существенно осложняется и требует использования новых технических приемов. Для доказательства теоремы 1 мы развиваем технику, которая в более частных случаях использовалась в работах \cite{NZSh1, BSh, NZSh2, Bel2, BelShk}. Помимо уже известных результатов мы доказываем еще три леммы, которые представляют и самостоятельный интерес. %из которых наиболее важной и технически сложной является лемма о дифференцировании мультипликаторов.

Для нестрихартцевского случая мы получаем следующий результат.

\begin{proclaim}{Теорема 2}
Пусть $1 < p \leqslant q$ и выполнено условие
$$
\left( \frac 1 p - \frac 1 q \right) \cdot n \; < \; s \; < \; \frac n {\max(p,q')} \: .
$$
Тогда имеют место непрерывные вложения
$$
H^{-s}_{r, \: unif}(\mathbb{R}^n) \subset M[H^s_p(\mathbb{R}^n) \to H^{-s}_q(\mathbb{R}^n)] \subset H^{-s}_{\max(p', q), \: unif}(\mathbb{R}^n),
$$
где
$$
r \; = \; \frac n {s - \left(\frac 1 p - \frac 1 q \right)n } \: .
$$
\end{proclaim}

\medskip

Статья имеет следующую структуру. В $\S \: 2$ мы вводим необходимые обозначения и определения, используемые в работе, а также приводим для удобства читателей формулировки результатов, которые применяются для доказательства основных теорем. В $\S \: 3$ мы доказываем три леммы, используемые в доказательстве основной теоремы при работе с наиболее трудным случаем, который рассматривается на шаге 6 доказательства.  Наконец, в $\S \: 4$ мы проводим доказательства обеих  теорем.

\section{ОСНОВНЫЕ ОПРЕДЕЛЕНИЯ И ФАКТЫ}

{\bf 1.} Далее будут использоваться следующие обозначения.

1) Через $\gamma'$ обозначается сопряженное с  $\gamma > 1$ число, определяемое соотношением
$$
1/\gamma + 1/\gamma' = 1;
$$

\medskip

2) в пространстве  $\mathbb R^n$ будем рассматривать две нормы
$$
| x |_1 \stackrel{def}{=} \sum\limits_{j = 1}^n |x_j|, \quad
| x | \stackrel{def}{=} \Bigl( \; \sum\limits_{j = 1}^n |x_j|^2 \: \Bigr)^{\frac{1}{2}}, \quad x = (x_1, \ldots, x_n) \in \mathbb{R}^n;
$$

\medskip

3) $\mu_L$ --- классическая мера Лебега на $\mathbb{R}^n;$

\medskip

4) $D(\mathbb{R}^n)$ --- пространство бесконечно дифференцируемых функций с компактным носителем, на котором топология определена стандартным образом;

\medskip

5) $S(\mathbb{R}^n)$ --- пространство Шварца (быстро убывающих бесконечно дифференцируемых функций) со стандартно определяемой топологией;

\medskip

6) $X'$ --- определяемое по топологическому линейному пространству $X$ пространство всех секвенциально непрерывных на $X$ комплекснозначных линейных функционалов с топологией поточечной сходимости;

\medskip

7) $\mathcal{B}(X, Y)$ --- пространство ограниченных линейных операторов из нормированного пространства $X$ в нормированное пространство $Y$ с обычной операторной нормой; при $X=Y$ это пространство обозначается $B(X)$;

\medskip

8) $X_{loc}(\mathbb{R}^n)$ --- определяемое по пространству $X(\mathbb{R}^n) \subset D'(\mathbb{R}^n)$ множество $\{ u \in D'(\mathbb{R}^n) | \; f \cdot u \in X(\mathbb{R}^n) \; \; \forall \: f \in D(\mathbb{R}^n) \};$

\medskip

9) $\mathbf{f}$ --- регулярный функционал из $D'(\mathbb{R}^n) \; \left( \mbox{или} \; S'(\mathbb{R}^n) \right)$ с плотностью $f \in L_{1, \: loc}(\mathbb{R}^n)$
%\stackrel{def}{=} \{ f \colon \mathbb{R}^n \to \mathbb{C} \: | \; \varphi \cdot f \in L_1(\mathbb{R}^n) \; \; %\forall \: \varphi \in D(\mathbb{R}^n) \}$
(обычно $\mathbf{f}$ и $f$  отождествляют, но нам часто будет удобнее их различать);

\medskip

10) $f_{(z)}$ --- сдвиг функции $f \colon \mathbb{R}^n \to \mathbb{C}$ на $z \in \mathbb{R}^n,$ определяемый соотношением
$$
f_{(z)}(x) = f(x - z) \; \; \; \forall \: x \in \mathbb{R}^n;
$$

\medskip

11) $X_1 \cap X_2$ --- пересечение двух нормированных пространств $X_1$ и $X_2$ с  нормой
$$
\| x \|_{X_1 \cap X_2} = \max (\| x \|_{X_1}, \| x \|_{X_2}) \; \; \forall \: x \in X_1 \cap X_2.
$$

\medskip

Через $C, \: C_j, \ j\in \mathbb N$, обозначаются различные константы, не зависящие от функций (или распределений), участвующих в оценках.

\medskip

Отметим, что для всех $f, \: g \in D(\mathbb{R}^n)$ распределение $f \cdot \: \mathbf{g} = g \cdot \: \mathbf{f} \in D'(\mathbb{R}^n)$ является регулярным функционалом с плотностью $f \cdot g.$

\medskip

{\bf 2.} Приведем определения функциональных пространств, фигурирующих в работе, и соответствующих пространств мультипликаторов.
%\begin{definition}\label{regular_functional}

При $p \geqslant 1$ нормированное пространство $H^0_p(\mathbb{R}^n)$ определяется как линейное пространство всех регулярных распределений $\mathbf{f} \in S'(\mathbb{R}^n)$ с плотностью $f \in L_p(\mathbb{R}^n),$ причем норма на нём задаётся соотношением
$$
\| \mathbf{f} \|_{H^0_p(\mathbb{R}^n)} = \| f \|_{L_p(\mathbb{R}^n)} \; \; \forall \: f \in L_p(\mathbb{R}^n).
$$
%\end{definition}

\begin{definition}\label{def_Bessel_potential}
{\sl Пусть $s \in \mathbb{R}, \: p \geqslant  1.$  Пространством бесселевых потенциалов $H^s_p(\mathbb{R}^n)$ называется линейное пространство
$$
\lbrace{ u \in S'(\mathbb{R}^n) \: | \; J_s(u) \in H^0_p(\mathbb{R}^n) \rbrace}
$$
с  нормой
$$
\| u \|_{H^s_p(\mathbb{R}^n)} \stackrel{def}{=} \| J_s (u) \|_{H^0_p(\mathbb{R}^n)}.
$$
Здесь линейный оператор $J_s \colon S'(\mathbb{R}^n) \longrightarrow S'(\mathbb{R}^n)$ определяется равенством
$$
J_s(u) \stackrel{def}{\:=\:} {\mathcal F}^{-1} ( \varphi_s \cdot {\mathcal F}(u)) \quad \forall \; u \in S'(\mathbb{R}^n),
$$
где $\mathcal F$ и ${\mathcal F}^{-1}$ ---  прямое и обратное преобразования Фурье в $S'(\mathbb{R}^n),$ а функция $\varphi_s \colon \mathbb{R}^n \longrightarrow \mathbb{R}$ определяется соотношением
$$
\varphi_s(x) = (1 + |x|^2)^{s/2}, \quad  x \in \mathbb{R}^n.
$$
}
\end{definition}

Отметим, что для произвольных индексов $s \in \mathbb{R}$ и $p > 1$ множество $\mathbf{D}(\mathbb{R}^n) \stackrel{def}{=} \{ \mathbf{f} \in D'(\mathbb{R}^n) | \; f \in D(\mathbb{R}^n) \}$ является всюду плотным в пространстве $H^s_p(\mathbb{R}^n).$ %относительно нормы $\| \cdot \|_{H^s_p(\mathbb{R}^n)}.$

%Нетрудно показать, что$$J_{2 k} = (Id - \Delta)^k \; \; \forall \: k \in \mathbb{Z}_+,$$
%где $\Delta \colon S'(\mathbb{R}^n) \to S'(\mathbb{R}^n)$ --- оператор Лапласа.

При $s \in \mathbb{R}$ и $p > 1$ можно ввести дуальное скалярное произведение
$$
< \cdot, \cdot >_s \colon H^{-s}_{p'}(\mathbb{R}^n) \times H^s_p(\mathbb{R}^n) \to \mathbb{C},
$$
определяемое  соотношением
$$
< u, v >_s \: = \: < J_{-s}(u), J_s(v) >_0 \: , \quad \  u \in H^{-s}_{p'}(\mathbb{R}^n), \ \ v \in H^s_p(\mathbb{R}^n),
$$
где
$$
< \mathbf{f}, \mathbf{g} >_0 \: = \: \int\limits_{\mathbb{R}^n} f(x) \cdot \overline{g(x)} \: d\mu_L(x),
 \quad\  f \in L_{p'}(\mathbb{R}^n), \ \  g \in L_p(\mathbb{R}^n).
$$
С помощью этого дуального скалярного произведения естественным образом устанавливается изометрический изоморфизм пространства $H^{-s}_{p'}(\mathbb{R}^n)$ и пространства $(H^s_p(\mathbb{R}^n))^{*},$ сопряжённого к  $H^s_p(\mathbb{R}^n).$

\begin{remark}\label{Bessel_differentiation}
{\sl Хорошо известно (и легко следует из определения),
что для произвольного мультииндекса $\alpha \in \mathbb{Z}^n_+$ и произвольных чисел $s \in \mathbb{R}, \; p > 1$ сужение дифференциального оператора $D^{\alpha} \colon S'(\mathbb{R}^n) \to S'(\mathbb{R}^n)$ на $H^s_p(\mathbb{R}^n)$ является ограниченным оператором из пространства $H^s_p(\mathbb{R}^n)$ в пространство $H^{s - |\alpha|_1}_p(\mathbb{R}^n).$
}
\end{remark}

\medskip

\begin{definition}\label{loc_unif}
{\sl Пусть $s \in \mathbb{R}, \; p > 1$ и функция $\eta \in D(\mathbb{R}^n)$ удовлетворяет условиям
\begin{equation}\label{def2}
0 \leqslant \eta(x) \leqslant 1 \; \; \forall \: x \in \mathbb{R}^n, \; \; \; \; \; \eta(x) = 1 \; \; \forall \: x \colon |x| \leqslant 1.
\end{equation}
Линейное пространство
\begin{equation}\label{def2_1}
\{ u \in H^s_{p, \: loc}(\mathbb{R}^n) | \; \sup_{z \in \mathbb{R}^n} \; \| \eta_{(z)} \cdot u \|_{H^s_p(\mathbb{R}^n)} \: < + \infty \}
\end{equation}
называется пространством равномерно локализованных бесселевых потенциалов и обозначается $H^s_{p, \: unif}(\mathbb{R}^n)$. Норма в этом пространстве определяется равенством
$$
\| u \|_{H^s_{p, \: unif}(\mathbb{R}^n)} \stackrel{def}{\: = \:} \sup_{z \in \mathbb{R}^n} \; \| \eta_{(z)} \cdot u \|_{H^s_p(\mathbb{R}^n)} \; \; \forall \: u \in H^s_{p, \: unif}(\mathbb{R}^n).
$$
}
\end{definition}

Отметим, что в определении пространства  $H^s_{p, \: unif}(\mathbb{R}^n)$ участвует функция $\eta$  и норма этого пространства зависит от выбора этой функции. Однако определяемое соотношением \eqref{def2_1} множество не зависит от выбора функции $\eta$ и при выборе любых двух функций $\eta_1$ и $\eta_2$, подчинённых условиям \eqref{def2}, соответствующие нормы эквивалентны. %Поэтому фактически это пространство от выбора $\eta$ не зависит.

\medskip

\begin{definition}\label{GeneralizedMultipliersDef}
{\sl Пусть $s, \: t \in \mathbb{R}, \; p, \: q > 1.$  Пространством мультипликаторов из $H^s_p(\mathbb{R}^n)$ в $H^t_q(\mathbb{R}^n)$  назовем линейное пространство
$$
\{ \mu \in H^t_{q, \: loc}(\mathbb{R}^n) \; | \; \; \exists \: C > 0 \; \colon \: \| f \cdot \mu \|_{H^t_q(\mathbb{R}^n)} \leqslant C \: \| \mathbf{f} \|_{H^s_p(\mathbb{R}^n)} \; \; \forall \: f \in D(\mathbb{R}^n) \}.
$$
Это пространство обозначим через $M[H^s_p(\mathbb{R}^n) \to H^t_q(\mathbb{R}^n)],$ а его норму определим равенством
$$
\| \mu \|_{M[H^s_p(\mathbb{R}^n) \to H^t_q(\mathbb{R}^n)]} = \inf \{ C > 0 \: | \; \| f \cdot \mu \|_{H^t_q(\mathbb{R}^n)} \leqslant C \: \| \mathbf{f} \|_{H^s_p(\mathbb{R}^n)} \; \; \forall \; f \in D(\mathbb{R}^n) \}.
$$
}
\end{definition}

Тем самым, мультипликатор $\mu \in M[H^s_p(\mathbb{R}^n) \to H^t_q(\mathbb{R}^n)]$ определяет оператор умножения на плотном множестве $\mathbf{D}(\mathbb R^n) \subset H^s_p(\mathbb{R}^n)$, который, в силу наличия априорной оценки, можно продолжить по непрерывности на всё пространство $H^s_p(\mathbb{R}^n)$.

\medskip

\begin{definition}\label{uniform_multipliers}
{\sl Пусть $s, \: t \in \mathbb{R}, \; p, \: q > 1$ и функция $\eta \in D(\mathbb{R}^n)$ удовлетворяет условиям
\eqref{def2}.
Определим пространство $M_{unif}[H^s_p(\mathbb{R}^n) \to H^t_q(\mathbb{R}^n)]$ как множество распределений $u \in D'(\mathbb{R}^n),$ для которых корректно определена и конечна норма
$$
\| u \|_{M_{unif}[H^s_p(\mathbb{R}^n) \to H^t_q(\mathbb{R}^n)]} \stackrel{def}{=} \sup\limits_{z \in \mathbb{R}^n} \{ \: \| \eta_{(z)} \cdot u \|_{M[H^s_p(\mathbb{R}^n) \to H^t_q(\mathbb{R}^n)]} \}.
$$
}
\end{definition}

Как и в предыдущем определении, от выбора функции $\eta$  рассматриваемое пространство фактически не зависит. При выборе двух различных функций, подчиненных условию \eqref{def2}, мы получаем одно и то же множество, а соответствующие нормы эквивалентны.

\medskip

{\bf 3.} Для удобства читателя  приведем формулировки теоремы вложения Соболева и теоремы об интерполяции для пространств бесселевых потенциалов, а также формулировки некоторых результатов из недавней работы авторов \cite{BelShk}, которые будут существенно использоваться при доказательстве основных теорем этой  работы.

Известен следующий результат (см., например, \cite[\S 2.8.1, замечание~2]{Tr}).

\begin{proclaim}{Теорема вложения Соболева}
{\sl Пусть $s, \: t \in \mathbb{R}, \; p, \: q > 1.$ Если $p \leqslant q$ и $s - n/p \geqslant t - n/q,$ то имеет место непрерывное вложение
$$
H^s_p(\mathbb{R}^n) \subset H^t_q(\mathbb{R}^n).
$$
При этом, если хотя бы одно из условий $p \leqslant q$ и $s - n/p \geqslant t - n/q$ не выполняется, то указанное вложение не имеет места.
}
\end{proclaim}

\medskip

\begin{remark}\label{unif_embedding}
{\sl Пусть $s, \: t \in \mathbb{R}, \; p, \: q > 1$. Легко видеть, что справедливо непрерывное вложение
$$
H^s_p(\mathbb{R}^n) \subset H^s_{p, \: unif}(\mathbb{R}^n).
$$
Кроме того, если имеет место один из следующих двух случаев:
$$
a) \: p \leqslant q, \; s - \frac{n}{p} \geqslant t - \frac{n}{q} \: ; \; \; b) \:  p \geqslant q, \; s \geqslant t;
$$
то справедливо непрерывное вложение
$$
H^s_{p, \: unif}(\mathbb{R}^n) \subset H^t_{q, \: unif}(\mathbb{R}^n).
$$
}
\end{remark}

\begin{remark}\label{em}
{\sl Пусть $s, \: t \in \mathbb{R}, \; p, \: q > 1.$ Тогда при любых $k, \: l \geqslant 0$ из теоремы вложения Соболева имеем
$$
H^{s + k}_p(\mathbb{R}^n) \subset H^s_p(\mathbb{R}^n) \; \mbox{и} \; H^t_q(\mathbb{R}^n) \subset H^{t - l}_q(\mathbb{R}^n),
$$
откуда следует справедливость непрерывного вложения
$$
M[H^s_p(\mathbb{R}^n) \to H^t_q(\mathbb{R}^n)] \subset M[H^{s + k}_p(\mathbb{R}^n) \to H^{t - l}_q(\mathbb{R}^n)].
$$
}
\end{remark}

\medskip

Известно, что шкала пространств бесселевых потенциалов замкнута относительно комплексной интерполяции. А именно, справедлива следующая теорема (см., например, \cite[\S 2.4.2, замечание 2d]{Tr}).

\begin{proclaim}{Теорема о комплексной интерполяции для пространств бесселевых потенциалов}
{\sl Пусть $s_0, \: s_1 \in \mathbb{R}, \; p_0, \: p_1 > 1, \; \theta \in (0, 1).$ Тогда
$$
[H^{s_0}_{p_0}(\mathbb{R}^n), H^{s_1}_{p_1}(\mathbb{R}^n)]_{\theta} = H^{s_{\theta}}_{p_{\theta}}(\mathbb{R}^n),
$$
где
$$
\frac{1}{p_{\theta}} = \frac{1 - \theta}{p_0} + \frac{\theta}{p_1} \: , \; \; s_{\theta} = (1 - \theta) \cdot s_0 + \theta \cdot s_1
$$
и $[\cdot, \, \cdot]_\theta$ --- результат применения метода комплексной интерполяции с показателем $\theta$ к соответствующей паре пространств.
}
\end{proclaim}

\begin{remark}\label{interpolation}
{\sl Пусть $s_0, \: s_1, \: t_0, \: t_1 \in \mathbb{R}$ и $\; p_0, \: p_1, \: q_0, \: q_1 > 1.$ Из сформулированной интерполяционной теоремы следует, что для произвольного оператора $A \colon S'(\mathbb{R}^n) \longrightarrow S'(\mathbb{R}^n),$ такого, что
$$
A  \in \mathcal{B}(H^{s_0}_{p_0}(\mathbb{R}^n), H^{t_0}_{q_0}(\mathbb{R}^n)) \; \; \mbox{и} \; \;  A \in  \mathcal{B}(H^{s_1}_{p_1}(\mathbb{R}^n), H^{t_1}_{q_1}(\mathbb{R}^n)),
$$
а нормы ограничений $A$ на соответствующие пространства равны $M_0$ и $M_1$, имеем
$$
A  \in \mathcal{B}(H^{s_{\theta}}_{p_{\theta}}(\mathbb{R}^n), H^{t_{\theta}}_{q_{\theta}}(\mathbb{R}^n)) \; \; \; \forall \; \theta \in (0, 1),
$$
причем норма этого оператора не превышает числа $M_0^{1-\theta}\, M_1^\theta$.
}
\end{remark}

\medskip

Далее приведём нужные результаты из работы авторов \cite{BelShk} и простые следствия из них.

\begin{lemma}{\cite[Утверждение 1]{BelShk}}\label{multipliers_symmetry}
Пусть $s, \: t \in \mathbb{R}, \; p, \: q > 1.$ Тогда имеет место совпадение пространств
$$
M[H^s_p(\mathbb{R}^n) \to H^{-t}_{q'}(\mathbb{R}^n)] = M[H^t_q(\mathbb{R}^n) \to H^{-s}_{p'}(\mathbb{R}^n)],
$$
причем нормы этих пространств равны.
\end{lemma}

\begin{lemma}{\cite[Утверждение 2]{BelShk}}\label{trivial_multipliers_embedding}
Пусть $s, \: t \in \mathbb{R}, \; p, \: q > 1.$ Тогда справедливо непрерывное вложение
\begin{equation}\label{simple}
M[H^s_p(\mathbb{R}^n) \to H^t_q(\mathbb{R}^n)] \subset H^t_{q, \: unif}(\mathbb{R}^n) \cap H^{-s}_{p', \: unif}(\mathbb{R}^n).
\end{equation}
\end{lemma}

\begin{lemma}{\cite[Лемма 2]{BelShk}}\label{equal_norm}
Пусть $s, \: t \in \mathbb{R}, \; 1 < p \leqslant q.$ Тогда имеет место совпадение пространств
$$
M[H^s_p(\mathbb{R}^n) \to H^t_q(\mathbb{R}^n)] = M_{unif}[H^s_p(\mathbb{R}^n) \to H^t_q(\mathbb{R}^n)],
$$
причем нормы этих пространств эквивалентны.
\end{lemma}

\begin{lemma}{\cite[Лемма 3]{BelShk}}\label{simultaneously}
{\sl Пусть $s, \: t \in \mathbb{R}, \; p, \: q > 1$ и $p \leqslant q$. Если для некоторых чисел $\gamma \in \mathbb{R}, \: r > 1$ справедливо непрерывное вложение
\begin{equation}\label{first_embedding}
H^{\gamma}_r(\mathbb{R}^n) \subset M[H^s_p(\mathbb{R}^n) \to H^t_q(\mathbb{R}^n)],
\end{equation}
то также справедливо непрерывное вложение
\begin{equation}\label{second_embedding}
H^{\gamma}_{r, \: unif}(\mathbb{R}^n) \subset M[H^s_p(\mathbb{R}^n) \to H^t_q(\mathbb{R}^n)].
\end{equation}
}
\end{lemma}
Отметим, что  обратное утверждение остаётся верным, даже если в условиях леммы \ref{simultaneously} отказаться от ограничения $p \leqslant q$. А именно, из справедливости при некоторых $\gamma \in \mathbb{R}, \: r > 1$ непрерывного вложения \eqref{second_embedding} следует справедливость непрерывного вложения \eqref{first_embedding}.

\begin{lemma}\label{equal_lemma_new}
{\sl Пусть $\gamma, \: s, \: t \geqslant 0$ и $p, \: q, \: r > 1.$ Тогда следующие условия эквивалентны:
\begin{align*}
1) & \: \exists \: C > 0 \colon \; \| f \cdot \mathbf{g} \|_{H^{\gamma}_r(\mathbb{R}^n)} \leqslant C \: \| \mathbf{f} \|_{H^s_p(\mathbb{R}^n)} \; \| \mathbf{g} \|_{H^t_q(\mathbb{R}^n)} \; \; \forall \; f, \: g \in D(\mathbb{R}^n);\\
2) & \; H^s_p(\mathbb{R}^n) \; \text{непрерывно вложено в} \; M[H^t_q(\mathbb{R}^n) \to H^{\gamma}_r(\mathbb{R}^n)];\\
3) & \; H^t_q(\mathbb{R}^n) \; \text{непрерывно вложено в} \; M[H^s_p(\mathbb{R}^n) \to H^{\gamma}_r(\mathbb{R}^n)];\\
4) & \; H^{-\gamma}_{r'}(\mathbb{R}^n) \; \text{непрерывно вложено в} \; M[H^s_p(\mathbb{R}^n) \to H^{-t}_{q'}(\mathbb{R}^n)].
\end{align*}
}\end{lemma}

\begin{proof}
Эквивалентность условий $1), \: 2)$ и $3)$ была доказана в ~\cite[Лемма 4]{BelShk}, а эквивалентность условий 1) и 4) доказана в \cite[Lemma 4]{Bel2}.
\end{proof}

\begin{lemma}\label{unif_embedding_criterion}
Пусть $\gamma, \: s, \: t \geqslant 0, \; p, \: q, \: r > 1$ и $p \leqslant q'$. Тогда непрерывное вложение
$$
H^{-\gamma}_{r', \: unif}(\mathbb{R}^n) \subset M[H^s_p(\mathbb{R}^n) \to H^{-t}_{q'}(\mathbb{R}^n)]
$$
имеет место в том и только в том случае, когда выполняется мультипликативная оценка
$$
\| f \cdot \mathbf{g} \|_{H^{\gamma}_r(\mathbb{R}^n)} \leqslant C_0 \: \| \mathbf{f} \|_{H^s_p(\mathbb{R}^n)} \; \| \mathbf{g} \|_{H^t_q(\mathbb{R}^n)} \; \; \forall \; f, \: g \in D(\mathbb{R}^n).
$$
\end{lemma}

\begin{proof}
Это утверждение является следствием Леммы \ref{simultaneously} и Леммы \ref{equal_lemma_new}.
\end{proof}

\begin{lemma}{\cite[Лемма 6]{BelShk}}\label{multi_inequality}
Пусть $p, \: q > 1, \; s, \: t \geqslant 0, \; s > n/p$ и выполняется одно из двух условий:
$$
a) \; p \leqslant q, \; \; s - \frac n p \geqslant t - \frac n q \; ; \; \; \; b) \; p \geqslant q, \; s \geqslant t.
$$
Тогда справедлива мультипликативная оценка
$$
\| f \cdot \mathbf{g} \|_{H^t_q(\mathbb{R}^n)} \leqslant C \: \| \mathbf{f} \|_{H^s_p(\mathbb{R}^n)} \; \| \mathbf{g} \|_{H^t_q(\mathbb{R}^n)} \;\; \forall \; f, \: g \in D(\mathbb{R}^n),
$$
где константа $C > 0$ не зависит от выбора $f$ и $g$.
\end{lemma}

\bigskip

\section{ТРИ ЛЕММЫ}

Докажем лемму о дифференцировании мультипликаторов, которая играет важную роль в доказательстве основного результата этой работы. Утверждение этой леммы известно в случае, когда мультипликатор действует между пространствами с положительными индексами гладкости (см., например, \cite[Lemma 3.2.4]{MShbook}). Более того, условие
$$
\mu \in M[H^{s - m}_p(\mathbb{R}^n) \to H^{t - m}_q(\mathbb{R}^n)],
$$
которое фигурирует ниже в формулировке леммы, оказывается излишним при $s \geqslant t \geqslant m$. Но в случае $t - m < 0$ мы не можем получить нужный результат без этого дополнительного условия. %Для сокращения записи далее для пространств мультипликаторов используем обозначения
%$$
%X^{s, k}_{t, m}:= M[H^{s - k}_p(\mathbb{R}^n) \to H^{t - m}_q(\mathbb{R}^n)].
%$$

\begin{lemma}\label{multipliers_differentiation}
Пусть $s, \: t \in \mathbb{R}, \: m \in \mathbb{Z}_{+}, \; p, \: q > 1$ и
\begin{equation}\label{mu}
\mu \in M[H^s_p(\mathbb{R}^n) \to H^t_q(\mathbb{R}^n)] \cap M[H^{s - m}_p(\mathbb{R}^n) \to H^{t - m}_q(\mathbb{R}^n)].
\end{equation}
Тогда
$$
D^{\alpha}(\mu) \in M[H^s_p(\mathbb{R}^n) \to H^{t - m}_q(\mathbb{R}^n)] \; \; \forall \: \alpha \in \mathbb{Z}_+^n: |\alpha|_1 \leqslant m,
$$
причём
$$
\| D^{\alpha}(\mu) \|_{M[H^s_p(\mathbb{R}^n) \to H^{t - m}_q(\mathbb{R}^n)]} \leqslant C \: \| \mu \|_{M[H^s_p(\mathbb{R}^n) \to H^t_q(\mathbb{R}^n)] \cap M[H^{s - m}_p(\mathbb{R}^n) \to H^{t - m}_q(\mathbb{R}^n)]}.
$$
%\begin{equation}\label{mu}
%\mu \in X^{s, 0}_{t, 0} \, \cap \, X^{s, m}_{t, m}.
%\end{equation}
%Тогда
%$$
%D^{\alpha}(\mu) \in X^{s, 0}_{t, m}  \; \; \forall \: \alpha \in \mathbb{Z}_+^n: |\alpha|_1 \leqslant m,
%$$
%причём
%\begin{equation}\label{dif}
%\| D^{\alpha}(\mu) \|_{X^{s, 0}_{t, m}} \leqslant C \: \| \mu \|_{X^{s, 0}_{t, 0}\, \cap \, X^{s, m}_{t, m}}
%\end{equation}
%где $C > 0 \: - \;$ некоторая константа, зависящая только от $s, \: t, \: p, \: q$ и $m.$
\end{lemma}

\begin{proof}
Достаточно провести доказательство леммы для мультииндексов $\alpha$, таких, что   $|\alpha|_1 = m$. При $|\alpha|_1 < m$ утверждение будет следовать из Замечания \ref{em} и Замечания \ref{interpolation}.

Для сокращения записи далее для пространств мультипликаторов будем использовать обозначение
$$
X^{s, k}_{t, m}(p, q) \stackrel{def}{=} M[H^{s - k}_p(\mathbb{R}^n) \to H^{t - m}_q(\mathbb{R}^n)],
$$
причём в дальнейшем мы будем опускать индексы $p$ и $q$, поскольку в ходе доказательства леммы они фиксированы.

Проведём доказательство леммы индукцией по индексу $m \: = \: | \alpha |_1.$

Пусть $m \: = \: | \alpha |_1  \: = \: 1$ и
$$
\mu \in X^{s, 0}_{t, 0} \cap X^{s, 1}_{t, 1}.
$$
Тогда $D^{\alpha} \: = \partial/\partial x_i $ при некотором $i \in \{ 1,...,n \}$ и, как легко видеть,
$$
f \cdot D^{\alpha}(\mu ) \: = \:  D^{\alpha}(f \cdot \mu) -  D^{\alpha}(f) \cdot \mu \; \; \; \forall \: f \in D(\mathbb{R}^n).
$$

В силу замечания \ref{Bessel_differentiation} для произвольных чисел $\gamma \in \mathbb{R}$ и $r > 1$ дифференциальный оператор $D^{\alpha} \colon H^{\gamma}_r(\mathbb{R}^n) \to H^{\gamma - 1}_r(\mathbb{R}^n)$ является ограниченным. Поскольку при $m = 1$ выполнено условие \eqref{mu}, то
$$
f \cdot D^{\alpha}(\mu ) \: = \: D^{\alpha}(f \cdot \mu) - D^{\alpha}(f) \cdot \mu \in H^{t - 1}_q(\mathbb{R}^n) \; \; \forall \: f \in D(\mathbb{R}^n),
$$
причём для произвольной функции $f \in D(\mathbb{R}^n)$ справедлива оценка
$$
\| f \cdot D^{\alpha}(\mu) \|_{H^{t - 1}_q(\mathbb{R}^n)} \leqslant \| D^{\alpha}(f \cdot \mu) \|_{H^{t - 1}_q(\mathbb{R}^n)} + \| D^{\alpha}(f) \cdot \mu \|_{H^{t - 1}_q(\mathbb{R}^n)} \leqslant
$$
$$
\leqslant \: (C_1 + C_2) \| \mu \|_{X^{s, 0}_{t, 0} \cap X^{s, 1}_{t, 1}}  \cdot \| \mathbf{f} \|_{H^s_p(\mathbb{R}^n)},
$$
где $C_1$  и $C_2$ --- нормы операторов
$$
D^{\alpha}:\, H^t_q(\mathbb{R}^n) \to H^{t - 1}_q(\mathbb{R}^n), \quad  D^{\alpha}: \, H^s_p(\mathbb{R}^n) \to  H^{s - 1}_p(\mathbb{R}^n).
$$

Из определения пространства мультипликаторов получаем, что
$$
D^{\alpha} (\mu) \in X^{s, 0}_{t, 1}
$$
и при $C_3 = C_1 + C_2$ справедлива оценка
$$
\| D^{\alpha} (\mu) \|_{X^{s, 0}_{t, 1}} \leqslant C_3 \: \| \mu \|_{X^{s, 0}_{t, 0} \cap X^{s, 1}_{t, 1}}.
$$
Это завершает доказательство леммы при $m = 1.$

Предположим, что утверждение уже доказано для всех мультииндексов $\alpha \in \mathbb{Z}^n_+,$ таких, что $| \alpha |_1 \leqslant m - 1.$ Рассмотрим произвольный мультииндекс $\alpha \in \mathbb{Z}^n_+,$ такой, что $| \alpha |_1 = m,$ и произвольное распределение %$\mu$,
$$
\mu \in X^{s, 0}_{t, 0} \cap X^{s, m}_{t, m}.
$$
Тогда распределение $D^{\alpha}(\mu)$ можно представить в виде
$$
D^{\alpha}(\mu) \: = \: D^{\alpha_2}(D^{\alpha_1}(\mu)),
$$
где $\alpha_1, \: \alpha_2 \in \mathbb{Z}^n_+, \; \alpha_1 + \alpha_2 = \alpha, \; | \alpha_1 |_1 = 1$ и $| \alpha_2 |_1 = m - 1.$

Согласно Замечанию \ref{interpolation} имеем
$$
\mu \in X^{s, \theta}_{t, \theta} \quad \forall \: \theta \in (0, m),
$$
причём
$$
\| \mu \|_{X^{s, \theta}_{t, \theta}} \leqslant \| \mu \|_{X^{s, 0}_{t, 0} \cap X^{s, m}_{t, m}}.
$$

В частности, полагая $\theta = 1$ и $\theta = m - 1$ соответственно, получаем
$$
\mu \in X^{s, 0}_{t, 0} \cap X^{s, 1}_{t, 1}
$$
и
$$
\mu \in X^{s, m - 1}_{t, m - 1} \cap X^{s, m}_{t, m} = X^{s - m + 1, 0}_{t - m + 1, 0} \cap X^{s - m + 1, 1}_{t - m + 1, 1}.
$$

Cогласно утверждению базы индукции имеем
$$
D^{\alpha_1}(\mu) \in X^{s, 0}_{t, 1} \cap X^{s - m + 1, 0}_{t - m + 1, 1},
$$
с оценкой норм
$$
\| D^{\alpha_1}(\mu)\|_{X^{s, 0}_{t, 1}} \leqslant C_4 \: \| \mu \|_{X^{s, 0}_{t, 0} \cap X^{s, 1}_{t, 1}}
$$
и
$$
\| D^{\alpha_1}(\mu) \|_{X^{s - m + 1, 0}_{t - m + 1, 1}} \leqslant  C_5 \: \| \mu \|_{X^{s - m + 1, 0}_{t - m + 1, 0} \cap X^{s - m + 1, 1}_{t - m + 1, 1}}.
$$
Из этих двух оценок получаем
$$
\| D^{\alpha_1}(\mu) \|_{X^{s, 0}_{t - 1, 0} \cap X^{s, m - 1}_{t - 1, m - 1}} \leqslant C_6 \cdot \| \mu \|_{X^{s, 0}_{t, 0} \cap X^{s, m}_{t, m}}.
$$

Согласно предположению индукции отсюда получаем
$$
D^{\alpha}(\mu) \:= \: D^{\alpha_2}(D^{\alpha_1}(\mu)) \in X^{s, 0}_{t - 1, m - 1} = X^{s, 0}_{t, m},
$$
с оценкой
\begin{equation*}
\| D^{\alpha}(\mu) \|_{X^{s, 0}_{t, m}} \leqslant C_7 \cdot \: \| D^{\alpha_1}(\mu) \|_{X^{s, 0}_{t - 1, 0} \cap X^{s, m - 1}_{t - 1, m - 1}} \leqslant C_8 \cdot \: \| \mu \|_{X^{s, 0}_{t, 0} \cap X^{s, m}_{t, m}}.
\end{equation*}

\end{proof}

Следующую лемму можно рассматривать как обобщение известного результата (см., например, \cite[с. 74]{McLean}) о представлении произвольного элемента пространства $W^{-k}_p(\mathbb{R}^n)$ в виде дивергенции $k-$го порядка некоторой вектор-функции с компонентами из пространства $W^0_p(\mathbb{R}^n)$.

\begin{lemma}\label{multipliers_representation}
Пусть $k \in \mathbb{N}.$ Тогда существуют непрерывные линейные операторы $A_0, \: A_{\beta}, \; \beta \in \mathbb{Z}_+^n, \: |\beta|_1 = k,$ действующие из $S'(\mathbb{R}^n)$ в $S'(\mathbb{R}^n),$ такие, что для произвольного распределения $u \in S'(\mathbb{R}^n)$ имеем
\begin{equation}\label{Predstav}
u = \sum\limits_{|\beta|_1 \leqslant k} D^{\beta}(A_{\beta}(u)),
\end{equation}
причём для произвольных чисел $s \in \mathbb{R}$ и $p > 1$ ограничения операторов $ A_{\beta}$ на $H^s_p(\mathbb{R}^n)$ являются ограниченными линейными операторами из пространства $H^s_p(\mathbb{R}^n)$ в пространство $H^{s+k}_p(\mathbb{R}^n).$
\end{lemma}
\begin{proof}
Из определения оператора $J_s$ следует, что
$$
J_{2 k} = (Id - \Delta)^k \; \; \forall \: k \in \mathbb{Z}_+,
$$
где $\Delta \colon S'(\mathbb{R}^n) \to S'(\mathbb{R}^n)$ --- оператор Лапласа,  а $Id \colon S'(\mathbb{R}^n) \to S'(\mathbb{R}^n)$ --- тождественный оператор.
Тогда
$$
u = J_{2 k}(J_{-2 k}(u)) = (Id - \Delta)^k(J_{- 2 k}(u)) \; \; \forall \: u \in S'(\mathbb{R}^n).
$$
Поэтому найдется  набор констант $\{ C_{\alpha} \: | \: \alpha \in \mathbb{Z}^n_+, \: |\alpha|_1 \leqslant 2 \: k \}$, такой, что для произвольного распределения $u \in S'(\mathbb{R}^n)$ справедливо представление
$$
u = \sum\limits_{|\alpha|_1 \leqslant 2 k} C_{\alpha} \cdot D^{\alpha}(J_{-2 k}(u)).
$$

Произвольный мультииндекс $\alpha \in \mathbb{Z}_+^n,$ такой, что $ |\alpha|_1 \leqslant 2k,$ можно представить в виде
$$
\alpha = \beta(\alpha) + \gamma(\alpha),
$$
где $\beta(\alpha), \gamma(\alpha) \in \mathbb{Z}^n_+, \; |\beta(\alpha)|_1 \leqslant k$ и $|\gamma(\alpha)|_1 \leqslant k.$ Поэтому $u$ можно записать  в виде
$$
u = \sum\limits_{|\beta|_1 \leqslant k} D^{\beta} \left( \sum\limits_{|\gamma|_1 \leqslant k} C_{\beta, \gamma} \cdot D^{\gamma}(J_{-2 k}(u)) \right) = :\, \sum\limits_{|\beta|_1 \leqslant  k} D^{\beta}(A_{\beta}(u)).
$$
%Вообще говоря, такое представление не является единственным,  набор констант $C_{\beta, \gamma}$ определяется %неоднозначно, но доказательство не зависит от выбора этих констант.

Из этого представления видно, что для завершения доказательства леммы достаточно показать, что операторы
$$
D^\gamma \circ J_{-2k}: \  H^s_p(\mathbb{R}^n) \to H^{s + k}_p(\mathbb{R}^n)
$$
являются ограниченными при условии $|\gamma|_1 \leqslant k$. Действительно, такой оператор представляется в виде композиции операторов
\begin{align*}
J_{-2k} : \ H^s_p(\mathbb{R}^n)\to H^{s + 2k}_p(\mathbb{R}^n),\\
D^\gamma : H^{s+2k}_p(\mathbb{R}^n)\to
H^{s+2k-|\gamma|_1}_p(\mathbb{R}^n),\\
Id:\ H^{s+2k-|\gamma|_1}_p(\mathbb{R}^n)\to H^{s+k}_p(\mathbb{R}^n).
\end{align*}
Первый из этих операторов по определению изометричен, второй ограничен в силу Замечания \ref{Bessel_differentiation}, а ограниченность третьего оператора следует из теоремы вложения Соболева при условии $|\gamma|_1\leqslant k$. %Тем самым лемма доказана.
\end{proof}

\medskip

Как следствие Лемм \ref{multipliers_differentiation} и \ref{multipliers_representation} получаем важный в дальнейшем результат.

\begin{lemma}\label{important}
Пусть $s, \: t \in \mathbb{R}, \: m \in \mathbb{Z}_{+}, \; p, \: q > 1$ и справедливо непрерывное вложение
\begin{multline}\label{a}
H^{m-t }_q(\mathbb{R}^n) \cap H^{m-s}_{p'}(\mathbb{R}^n) \\
 \subset M[H^{s}_p(\mathbb{R}^n) \to H^{m-t}_q(\mathbb{R}^n)] \cap M[H^{s-m}_p(\mathbb{R}^n) \to H^{-t}_q(\mathbb{R}^n)].
\end{multline}
Тогда справедливо непрерывное вложение
\begin{equation}\label{b}
 H^{-t}_q(\mathbb{R}^n) \cap H^{-s}_{p'}(\mathbb{R}^n) \subset M[H^{s}_p(\mathbb{R}^n) \to H^{-t}_q(\mathbb{R}^n)].
\end{equation}
\end{lemma}

\begin{proof}
%Для сокращения записи далее используем обозначение
%$$
%Y_{\{k,l\}}:= M[H^{s-k}_p(\mathbb{R}^n) \to H^{l-t}_q(\mathbb{R}^n)].
%$$

Пусть
$$
\mu \in H^{-t}_q(\mathbb{R}^n) \cap H^{-s}_{p'}(\mathbb{R}^n).
$$
Тогда в силу Леммы \ref{multipliers_representation} справедливо представление
$$
\mu = \sum\limits_{|\beta|_1 \leqslant m} D^{\beta}(A_{\beta}(\mu)), \quad A_{\beta}(\mu) \in H^{m-t}_q(\mathbb{R}^n) \cap H^{m-s}_{p'}(\mathbb{R}^n),
$$
причем операторы
\begin{equation}\label{c}
A_\beta :  H^{-t}_q(\mathbb{R}^n) \to  H^{m-t}_q(\mathbb{R}^n) \ \ \text{и}\ \
A_\beta :  H^{-s}_{p'}(\mathbb{R}^n) \to  H^{m-s}_{p'}(\mathbb{R}^n)
\end{equation}
ограничены. При $|\beta|_1 \leqslant m$, пользуясь Леммой \ref{multipliers_differentiation}, непрерывностью вложения \eqref{a} и ограниченностью операторов \eqref{c}, получаем
$$
D^\beta(A_\beta(\mu)) \in M[H^s_p(\mathbb{R}^n) \to H^{-t}_q(\mathbb{R}^n)]
$$
и
\begin{multline*}
\|D^\beta(A_\beta(\mu))\|_{M[H^s_p(\mathbb{R}^n) \to H^{-t}_q(\mathbb{R}^n)]} \leqslant \\
\leqslant C_1 \cdot \|A_\beta(\mu)\|_{M[H^s_p(\mathbb{R}^n) \to H^{m - t}_q(\mathbb{R}^n)] \cap M[H^{s - m}_p(\mathbb{R}^n) \to H^{-t}_q(\mathbb{R}^n)]} \leqslant \\
\leqslant C_2 \cdot \|A_\beta(\mu)\|_{H^{m-s}_{p'}(\mathbb{R}^n) \cap H^{m-t}_q(\mathbb{R}^n)} \leqslant C_3 \cdot \|\mu\|_{H^{-s}_{p'}(\mathbb{R}^n) \cap H^{-t}_q(\mathbb{R}^n)}.
\end{multline*}
Поскольку распределение $\mu$ представляется конечной суммой распределений $D^{\beta}(A_\beta(\mu))$, то справедливо вложение \eqref{b} с оценкой норм.
\end{proof}

\section{ДОКАЗАТЕЛЬСТВА ОСНОВНЫХ ТЕОРЕМ}

{\bf Доказательство теоремы 1.} Первое утверждение теоремы  о непрерывном вложении \eqref{emb} следует из леммы \ref{trivial_multipliers_embedding}. Основная трудность состоит в доказательстве обратного непрерывного вложения
\begin{equation}\label{4.1}
H^{-t}_{q, \: unif}(\mathbb{R}^n) \cap H^{-s}_{p', \: unif}(\mathbb{R}^n) \subset M[H^s_p(\mathbb{R}^n) \to H^{-t}_{q}(\mathbb{R}^n)].
\end{equation}

{\it Шаг 1.} Достаточно доказать справедливость непрерывного вложения \eqref{4.1} в случае $s\geqslant t\geqslant 0, \ s> n/p$. Случай $t\geqslant s\geqslant 0, \ t>n/q'$  сводится к предыдущему. Действительно, согласно лемме \ref{multipliers_symmetry} имеет место совпадение пространств
$$
M[H^s_p(\mathbb{R}^n) \to H^{-t}_q(\mathbb{R}^n)] = M[H^t_{q'}(\mathbb{R}^n) \to H^{-s}_{p'}(\mathbb{R}^n)]
$$
и равенство их норм, а неравенства $p \leqslant q$ и $q' \leqslant p'$ равносильны. Поэтому второй случай сводится к первому заменой чисел $s$ и $t$, $p$  и $q'$.

{\it Шаг 2.} Итак, пусть выполнены условия
\begin{equation}\label{10}
s\geqslant t \geqslant 0, \ \  s > n/p,\ \ p \leqslant q.
\end{equation}
В этой ситуации  возможны три случая
\begin{equation}\label{11}
p\geqslant q';
\end{equation}
\begin{equation}\label{12}
p<q' \ \ \text{и} \ \ s-n/p \geqslant t-n/q';
\end{equation}
\begin{equation}\label{13}
p<q' \ \  \text{и} \ \  s-n/p < t-n/q'.
\end{equation}
Рассмотрим отдельно каждый из этих трех случаев.

{\it Шаг 3.} В случае \eqref{11} имеем $q\geqslant p'$  и $-t \geqslant -s$. Тогда, согласно Замечанию \ref{unif_embedding}, имеем непрерывное вложение
\begin{equation}\label{14}
H^{-t}_{q, \: unif}(\mathbb{R}^n) \subset H^{-s}_{p', \: unif}(\mathbb{R}^n).
\end{equation}
Следовательно,
\begin{equation}\label{15}
H^{-t}_{q, \: unif}(\mathbb{R}^n) \cap H^{-s}_{p', \: unif}(\mathbb{R}^n) = H^{-t}_{q, \: unif}(\mathbb{R}^n),
\end{equation}
причём нормы этих пространств эквивалентны.
Так как $p \leqslant q,$ то в силу Леммы \ref{unif_embedding_criterion} непрерывное вложение
\begin{equation}\label{16}
H^{-t}_{q, \: unif}(\mathbb{R}^n) \subset M[H^s_p(\mathbb{R}^n) \to H^{-t}_{q}(\mathbb{R}^n)]
\end{equation}
имеет место тогда и только тогда, когда справедлива мультипликативная оценка
\begin{equation}\label{17}
\| f \cdot \mathbf{g} \|_{H^t_{q'}(\mathbb{R}^n)} \leqslant C \: \| \mathbf{f} \|_{H^s_p(\mathbb{R}^n)} \| \mathbf{g} \|_{H^t_{q'}(\mathbb{R}^n)} \; \; \; \forall \; f, \: g \in D(\mathbb{R}^n).
\end{equation}
Но справедливость этой оценки при $p\geqslant q'$ следует из Леммы \ref{multi_inequality} (реализуется случай b) в этой лемме). Поэтому справедливость непрерывного вложения \eqref{4.1} вытекает из \eqref{15} и \eqref{16}.

{\it Шаг 4.}
В случае \eqref{12} доказательство вложения \eqref{4.1}  проводится аналогично. В этом случае выполняются неравенства
$$
q < p' \ \ \text{и}\ \ -t-n/q \geqslant -s -n/p'.
$$
Тогда в силу Замечания \ref{unif_embedding} имеем \eqref{14} и \eqref{15}, а в силу Леммы \ref{unif_embedding_criterion} справедливость непрерывного вложения \eqref{16} равносильна выполнению мультипликативной оценки \eqref{17}. Но, в силу выполнения условий \eqref{12}, справедливость оценки \eqref{17} также следует из Леммы \ref{multi_inequality} (реализуется случай a) в этой лемме).

{\it Шаг 5.} Остается рассмотреть наиболее трудный случай \eqref{13}. Напомним, что в этом случае выполнены условия
\begin{equation}\label{18}
p < q', \ \ s - n/p < t - n/q', \ \ s \geqslant t \geqslant 0, \ \ s > n/p, \ \ p < q.
\end{equation}

Покажем, что для доказательства непрерывного вложения \eqref{4.1} достаточно доказать справедливость непрерывного вложения
\begin{equation}\label{40}
H^{-t}_q(\mathbb{R}^n) \, \cap \, H^{-s}_{p'}(\mathbb{R}^n) \subset M[H^s_p(\mathbb{R}^n) \to H^{-t}_q(\mathbb{R}^n)].
\end{equation}
Это утверждение аналогично Лемме 4, но применить здесь эту лемму непосредственно не удаётся. Поэтому  приведем независимое доказательство.

Пусть справедливо непрерывное вложение \eqref{40} и $\eta \in D(\mathbb{R}^n)$ --- функция, удовлетворяющая условиям \eqref{def2} из Определения \ref{loc_unif}.

Рассмотрим произвольное распределение $\mu \in H^{-t}_{q, \: unif}(\mathbb{R}^n) \, \cap \, H^{-s}_{p', \: unif}(\mathbb{R}^n).$ Тогда для произвольного $z \in \mathbb{R}^n$ получаем
$$
\eta_{(z)} \cdot \mu \in H^{-t}_q(\mathbb{R}^n) \cap H^{-s}_{p'}(\mathbb{R}^n) \subset M[H^s_p(\mathbb{R}^n) \to H^{-t}_q(\mathbb{R}^n)],
$$
причём
$$
\| \eta_{(z)} \cdot \mu \|_{M[H^s_p(\mathbb{R}^n) \to H^{-t}_q(\mathbb{R}^n)]} \leqslant C_1 \| \eta_{(z)} \cdot \mu \|_{H^{-t}_q(\mathbb{R}^n) \cap H^{-s}_{p'}(\mathbb{R}^n)}
$$
$$
\leqslant C_1 \| \mu \|_{H^{-t}_{q, \: unif}(\mathbb{R}^n) \cap H^{-s}_{p', \: unif}(\mathbb{R}^n)},
$$
где константа $C_1 > 0$ не зависит от выбора $z \in \mathbb{R}^n.$

Переходя к супремуму по $z \in \mathbb{R}^n$ в этой оценке и используя Лемму \ref{equal_norm} (ее условия выполнены, поскольку $p \leqslant q$), получаем
$$
\mu \in M_{unif}[H^s_p(\mathbb{R}^n) \to H^{-t}_q(\mathbb{R}^n)] = M[H^s_p(\mathbb{R}^n) \to H^{-t}_q(\mathbb{R}^n)]
$$
и
$$
\| \mu \|_{M[H^s_p(\mathbb{R}^n) \to H^{-t}_q(\mathbb{R}^n)]} \leqslant C_2 \| \mu \|_{M_{unif}[H^s_p(\mathbb{R}^n) \to H^{-t}_q(\mathbb{R}^n)]} \leqslant
$$
$$
\leqslant C_3 \| \mu \|_{H^{-t}_{q, \: unif}(\mathbb{R}^n) \cap H^{-s}_{p', \: unif}(\mathbb{R}^n)}.
$$
Это и означает справедливость непрерывного вложения \eqref{4.1}.

{\it Шаг 6.}
Итак, нам осталось доказать лишь справедливость непрерывного вложения \eqref{40}. Для этого докажем сначала, что имеют место непрерывные вложения
\begin{equation}\label{41}
H^{n-t}_q(\mathbb{R}^n) \subset M[H^s_p(\mathbb{R}^n) \to H^{n-t}_q(\mathbb{R}^n)],
\end{equation}

\begin{equation}\label{42}
H^{n-s}_{p'}(\mathbb{R}^n) \subset M[H^{s-n}_p(\mathbb{R}^n) \to H^{-t}_q(\mathbb{R}^n)].
\end{equation}
Тогда будет доказана справедливость непрерывного вложения
\begin{multline*}%\label{43}
H^{n-t}_q(\mathbb{R}^n) \, \cap\, H^{n-s}_{p'}(\mathbb{R}^n) \\
 \subset
M[H^s_p(\mathbb{R}^n) \to H^{n-t}_q(\mathbb{R}^n)]\ \cap\  M[H^{s-n}_p(\mathbb{R}^n) \to H^{-t}_q(\mathbb{R}^n)].
\end{multline*}
Из Леммы \ref{important} (используем эту лемму при $m = n$) мы получим тогда справедливость непрерывного вложения \eqref{40}.

Таким образом, для завершения доказательства основного утверждения теоремы нам достаточно доказать справедливость непрерывных вложений \eqref{41} и \eqref{42}.

{\it Шаг 7.} Докажем справедливость непрерывного вложения \eqref{41}. Рассмотрим сначала случай, когда $n - t \geqslant 0.$ В силу Леммы \ref{equal_lemma_new} (используем эквивалентность условий 1) и 3) в этой лемме) справедливость непрерывного вложения \eqref{41} равносильна выполнению мультипликативной оценки
\begin{equation}\label{21}
\| f \cdot \mathbf{g} \|_{H^{n - t}_q(\mathbb{R}^n)} \leqslant \: C_4 \: \| \mathbf{f} \|_{H^s_p(\mathbb{R}^n)} \: \| \mathbf{g} \|_{H^{n - t}_q(\mathbb{R}^n)} \; \; \; \forall \; f, \: g \in D(\mathbb{R}^n).
\end{equation}
Из \eqref{18} следует
\begin{equation}\label{21b}
t> n/q'\ \ \text{и} \ \ s+t > n/p +n/q',
\end{equation}
поэтому с учетом других условий \eqref{18} имеем
\begin{equation}\label{21a}
 n - t \geqslant 0, \ \ s > n/p, \ \ p < q, \ \ s - n/p > n - t - n/q.
\end{equation}
Это влечет выполнение условий Леммы \ref{multi_inequality} (реализуется случай а) в этой лемме). Следовательно, мультипликативная оценка \eqref{21} выполняется, а значит, при $n - t \geqslant 0$  справедливо непрерывное вложение \eqref{41}.

Пусть теперь $n - t < 0$. Тогда, согласно Лемме \ref{equal_lemma_new} (используем эквивалентность условий 1) и 4) в этой лемме), справедливость непрерывного вложения \eqref{41} равносильна выполнению мультипликативной оценки
\begin{equation}\label{22}
\| f \cdot \mathbf{g} \|_{H^{t - n}_{q'}(\mathbb{R}^n)} \leqslant C_5 \: \| \mathbf{f} \|_{H^s_p(\mathbb{R}^n)} \| \mathbf{g} \|_{H^{t - n}_{q'}(\mathbb{R}^n)} \; \;  \; \forall \; f, \: g \in D(\mathbb{R}^n).
\end{equation}
Но эта оценка также справедлива согласно Лемме \ref{multi_inequality}. Действительно, реализуется случай a) в этой лемме, поскольку
$$
t - n > 0, \; s > n/p, \; p \leqslant q' \ \ \text{и} \ \ s - n/p \geqslant t - n - n/q'.
$$
Тем самым непрерывность вложения \eqref{41} доказана вне зависимости от знака $n - t$.

{\it Шаг 8.} Докажем теперь справедливость непрерывного вложения \eqref{42}. Согласно Лемме \ref{multipliers_symmetry} справедливо равенство
$$
M[H^{s-n}_p(\mathbb{R}^n) \to H^{-t}_q(\mathbb{R}^n)] =  M[H^t_{q'}(\mathbb{R}^n) \to H^{n-s}_{p'}(\mathbb{R}^n)],
$$
поэтому достаточно доказать непрерывное вложение
\begin{equation}\label{last}
H^{n-s}_{p'}(\mathbb{R}^n) \subset M[H^t_{q'}(\mathbb{R}^n) \to H^{n-s}_{p'}(\mathbb{R}^n)].
\end{equation}

Так же, как на Шаге 6, нужно рассмотреть случаи разных знаков числа $n - s$. Пусть  сначала $n - s \geqslant 0$. В силу Леммы \ref{equal_lemma_new} (используем эквивалентность условий 1) и 3)) справедливость непрерывного вложения \eqref{last} равносильна выполнению мультипликативной оценки
\begin{equation}\label{22}
\| f \cdot \mathbf{g} \|_{H^{n - s}_{p'}(\mathbb{R}^n)} \leqslant \: C_6 \: \| \mathbf{f} \|_{H^t_{q'}(\mathbb{R}^n)} \: \| \mathbf{g} \|_{H^{n - s}_{p'}(\mathbb{R}^n)} \; \; \; \forall \; f, \: g \in D(\mathbb{R}^n).
\end{equation}
Неравенство $p < q$ влечет $q' < p'$,  поэтому из \eqref{18} и \eqref{21b} имеем %по аналогии с \eqref{21a} имеем
$$
 n - s \geqslant 0, \ \ t > n/q', \ \ q' < p' \ \ \text{и} \ \ t - n/q' \geqslant n - s - n/p'.
$$
Следовательно, выполнены условия  Леммы \ref{multi_inequality} (реализуется случай a) в этой лемме). Значит, верна  оценка \eqref{22}, которая влечет  справедливость непрерывного вложения \eqref{last}.

Рассмотрим теперь случай $n - s < 0$. Тогда, согласно Лемме \ref{equal_lemma_new} (используем эквивалентность условий 1) и 4) в этой лемме), справедливость непрерывного вложения \eqref{last} равносильна мультипликативной оценке
\begin{equation}\label{22.1}
\| f \cdot \mathbf{g} \|_{H^{s - n}_p(\mathbb{R}^n)} \leqslant C_7 \: \| \mathbf{f} \|_{H^t_{q'}(\mathbb{R}^n)} \| \mathbf{g} \|_{H^{s - n}_p(\mathbb{R}^n)} \; \;  \; \forall \; f, \: g \in D(\mathbb{R}^n).
\end{equation}
В этом случае, согласно \eqref{18} и \eqref{21b}, имеем
$$
t \geqslant s - n > 0, \ \ q' > p, \ \ t > n/q'.
$$
Следовательно, выполнены условия  Леммы \ref{multi_inequality} (реализуется случай b) в этой лемме). Поэтому верна оценка \eqref{22}, которая влечет справедливость непрерывного вложения \eqref{last}.

Таким образом, справедливость непрерывных вложений \eqref{last} и \eqref{42} доказана вне зависимости от знака $n - s$.

Тем самым Теорема 1 полностью доказана.

%{\it Шаг 8.}  Приведенные в Теореме 1  условия на индексы при $p, q > 1$  являются необходимыми для выполнения %равенства \eqref{equi}.
%Действительно, известно \cite[Remark 4.9.2]{RSbook}, что в случае невыполнения неравенства $s > n/p$  при %$s\geqslant t$  или $t> n/q'$ при $t\geqslant s$ равенство \eqref{equi} невозможно.
%Далее, заметим, что регулярный функционал $\mathbf{I},$ порождённый функцией, тождественно равной единице на %$\mathbb{R}^n,$ принадлежит $H^{\gamma}_{r, \: unif}(\mathbb{R}^n)$ для произвольных индексов $\gamma \in %\mathbb{R}$ и $r > 1$.  Поэтому если справедливо непрерывное вложение
%$$
%H^{\gamma}_{r, \: unif}(\mathbb{R}^n) \subset M[H^s_p(\mathbb{R}^n) \to H^t_q(\mathbb{R}^n)],
%$$
%то также справедливо непрерывное вложение
%$$
%H^s_p(\mathbb{R}^n) \subset H^{-t}_q(\mathbb{R}^n).
%$$
%Но, как отмечено в формулировке теоремы Соболева, необходимым условием справедливости этого непрерывного вложения %явлется выполнение условия $p \leqslant q$

\begin{remark}\label{necessity}
Пусть $s, t \geqslant 0$.  Тогда ограничения на индексы $s, t, p, q$, накладывемые в формулировке Теоремы 1, не только достаточны для выполнения равенства \eqref{equi}, но и необходимы. Авторы обладают доказательством этого утверждения, если исключить из рассмотрения пограничные случаи $s = n/p$ или $t = n/q'$, в которых остаются нерешённые вопросы.
\end{remark}

%В ситуации, когда индексы гладкости пространств различного знака, при отказе от выполнения любого из условий %теоремы \ref{main_multipliers_theorem} невозможно в общем случае дать характеризацию пространства мультипликаторов %из одного пространства бесселевых потенциалов в другое в терминах шкалы пространств $H^{\gamma}_{r, \: %unif}(\mathbb{R}^n).$

%\underbrace{}Необходимость условия %$p \leqslant q'$ получается с помощью рассуждений, аналогичных проведённым в %\cite[Замечание 5]{BelShk}.

%Отказ же от условий 1) и 2), обобщающих условие Стрихартца, невозможен даже в модельном случае $p = q = 2$ . В %самом деле, при $s, \: t \in (0, n/2),$ с одной стороны, известно (см. \cite[Lemma 3, Lemma 6]{NZSh2}), что имеет %место цепочка непрерывных вложений
%$$
%H^{-\min(s, t)}_{\frac{n}{\max(s, t)}, \: unif}(\mathbb{R}^n) \subset M[H^s_2(\mathbb{R}^n) \to %H^{-t}_2(\mathbb{R}^n)] \subset H^{-\min(s, t)}_{2, \: unif}(\mathbb{R}^n),
%$$
%а, с другой стороны, можно показать, что для произвольного числа $\varepsilon > 0$ существует распределение
%$$
%u_{\varepsilon} \in H^{- \min(s, t)}_{\frac{n}{\max(s, t)} - \varepsilon, \; unif}(\mathbb{R}^n) \setminus %M[H^s_2(\mathbb{R}^n) \to H^{-t}_2(\mathbb{R}^n)].
%$$
%Отсюда несложно получить, что если $s, \: t \in (0, n/2),$ то не существует числа $r \in [2, n/\max(s, t)],$ %такого, что
%$$
%M[H^s_2(\mathbb{R}^n) \to H^{-t}_2(\mathbb{R}^n)] = H^{-\min (s, t)}_{r, unif}(\mathbb{R}^n).
%$$
%\end{remark}

\bigskip

%Отметим также, что теорема \ref{main_multipliers_theorem} и \cite[Теорема 1]{BelShk}, с учётом замечания %\ref{mixed_sign_multipliers_necessity} и \cite[Замечание 5]{BelShk}, дают решение проблемы описания пространства %мультипликаторов $M[H^s_p(\mathbb{R}^n) \to H^t_q(\mathbb{R}^n)]$ в терминах шкалы равномерно локализованных %пространств бесселевых потенциалов при $s, \: t \in \mathbb{R}, \; p, \: q > 1,$ то есть в ситуации, когда индексы %гладкости $s$ и $t$ произвольного знака.

%Действительно, когда оба индекса гладкости отрицательны, вопрос сводится к ситуации, рассмотренной в %\cite{BelShk}, поскольку при $s, \: t \geqslant 0, \; p, \: q > 1$
%$$
%M[H^{-s}_p(\mathbb{R}^n) \to H^{-t}_q(\mathbb{R}^n)] = M[H^t_{q'}(\mathbb{R}^n) \to H^s_{p'}(\mathbb{R}^n)].
%$$
%С другой стороны, при $s, \: t \in \mathbb{R}, \; p, \: q > 1$ для получения описания пространства %мультипликаторов $M[H^s_p(\mathbb{R}^n) \to H^t_q(\mathbb{R}^n)]$ в терминах равномерно локализованных пространств %бесселевых потенциалов необходимо выполнение неравенств
%$$
%s - \frac{n}{p} \geqslant t - \frac{n}{q} \; \; \mbox{и} \; \; p \leqslant q,
%$$
%из которых следует, что
%$$
%s \geqslant t + \frac{n}{p} - \frac{n}{q} \geqslant t.
%$$
%В частности, отсюда мы получаем невозможность такого описания для пространства $M[H^s_p(\mathbb{R}^n) \to %H^t_q(\mathbb{R}^n)]$ при $s < 0, \: t \geqslant 0.$

%\bigskip

Если условия стрихартцевского типа не выполняются, то при естественных дополнительных ограничениях оказывается возможным установить справедливость непрерывного вложения некоторого равномерно локализованного пространства бесселевых потенциалов в пространство мультипликаторов.
Мы докажем соответствующий результат для мультипликаторов, действующих между пространствами бесселевых потенциалов с индексами гладкости разного знака, в наиболее важном случае, когда  абсолютные значения индексов гладкости этих пространств совпадают.

\medskip

{\bf Доказательство теоремы 2.}
Справедливость непрерывного вложения
$$
M[H^s_p(\mathbb{R}^n) \to H^{-s}_q(\mathbb{R}^n)] \subset H^{-s}_{\gamma, \: unif}(\mathbb{R}^n), \quad \gamma:= \max (p', q),
$$
следует непосредственно из Леммы \ref{trivial_multipliers_embedding} и Замечания \ref{unif_embedding}.

Пусть выполнены условия Теоремы 2, т.е. $p\leqslant q$,
$$
\left(\frac 1p - \frac 1 q \right) n < s < \frac n{\max (p,q')}\quad \text{и}\ \ r: =\frac n{s-\left(\frac 1p -\frac 1q\right)n}.
$$
Заметим, что в этом случае
$$
r\geqslant \frac ns > \max (p, q') > 1.
$$
%число $r$, определённое в формулировке Теоремы 2, удовлетворяет неравенству $r > 1$, поскольку в силу условий теоремы
%$$
%r \; = \; \frac n {s - \left(\frac 1 p - \frac 1 q \right)n } \; \geqslant \; \frac{n}{s} \; > \; \max(p, q') \; > \; 1.
%$$
Так как $p \leqslant q,$ то можно применить Лемму \ref{unif_embedding_criterion}, согласно которой непрерывное вложение
$$
H^{-s}_{r, \: unif}(\mathbb{R}^n) \subset M[H^s_p(\mathbb{R}^n) \to H^{-s}_q(\mathbb{R}^n)]
$$
имеет место тогда и только тогда, когда выполняется  мультипликативная оценка
\begin{equation}\label{new_estimate}
\| f \cdot \mathbf{g} \|_{H^s_{r'}(\mathbb{R}^n)} \leqslant C \: \| \mathbf{f} \|_{H^s_p(\mathbb{R}^n)} \: \| \mathbf{g} \|_{H^s_{q'}(\mathbb{R}^n)} \; \; \forall \: f, \: g \in D(\mathbb{R}^n).
\end{equation}

Согласно \cite[Theorem 4.2.1]{SiTr},  эта оценка выполняется, если
%справедливая при некоторой константе $C_1 > 0$ мультипликативная оценка
%$$
%\| f \cdot \mathbf{g} \|_{H^s_{r_0}(\mathbb{R}^n)} \leqslant C_1 \: \| \mathbf{f} \|_{H^s_p(\mathbb{R}^n)} \: \| \mathbf{g} \|_{H^s_{q'}(\mathbb{R}^n)} \; \; %\forall \: f, \: g \in D(\mathbb{R}^n)
%$$
%для $r_0,$ определённого соотношением
\begin{equation}\label{36}
\frac{1}{r'} = \frac{1}{p} + \frac{1}{q'} - \frac{s}{n} \: ,
\end{equation}
а также  выполняются еще следующие три  условия:
\begin{equation}\label{37}
\frac{1}{p} - \frac{s}{n} > 0, \; \; \frac{1}{q'} - \frac{s}{n} > 0 \; \; \mbox{и} \; \; \left(\frac{1}{p} - \frac{s}{n}\right) + \left(\frac{1}{q'} - \frac{s}{n}\right) < 1.
\end{equation}
 (В действительности, требование третьего условия в \eqref{37} является излишним, так как это условие заведомо выполняется в случае равенства
 \eqref{36} и неравенства $r' >1$).   Проверим выполнение всех этих условий. Равенство \eqref{36} эквивалентно  равенству
$$
\frac 1r  \; = \; \frac 1q - \frac 1p  + \frac sn \: ,
$$
которое вытекает из определения числа $r$. Далее, первые два неравенства в \eqref{37}  следуют непосредственно из условия
теоремы
$$
s \; < \; \frac n {max(p, q')} \: .
$$
%а третье неравенство в \eqref{37} следует из условия
%$$
%s >\left(\frac 1p -\frac 1q\right) n \: .
%$$

Таким образом, имеет место мультипликативная оценка \eqref{new_estimate}, а значит, справедливо непрерывное вложение
$$
H^{-s}_{r, \: unif}(\mathbb{R}^n) \subset M[H^s_p(\mathbb{R}^n) \to H^{-s}_q(\mathbb{R}^n)].
$$

Теорема 2 доказана.

\bigskip
\bigskip

%\author{}

%Беляев А.А.

%Московский государственный университет
%им. М.~В.~Ломоносова

%механико-математический факультет

%119991, Москва
%Ленинские горы, 1, Главное здание МГУ
%Россия

%e-mail: alexei.a.belyaev@gmail.com

%\bigskip

%Шкаликов А.А.

%Московский государственный университет
%им. М.~В.~Ломоносова

%механико-математический факультет

%119991, Москва
%Ленинские горы, 1, Главное здание МГУ
%Россия

%e-mail: shkalikov@mi.ras.ru


\begin{thebibliography}{9}
\bibitem{BSh}
Бак~Дж.~Г., Шкаликов~А.~А.,
\emph{Мультипликаторы в дуальных соболевских пространствах и операторы Шрёдингера с потенциалами-распределениями},
Матем. заметки, {\bf 71} (2002), № 5, 643--651.

\bibitem{BelShk}
Беляев~A.~А., Шкаликов~А.~А.,
\emph{Мультипликаторы в пространствах бесселевых потенциалов: случай индексов неотрицательной гладкости},
Матем. заметки, {\bf 102} (2017), № 5, 684--699.


\bibitem{Bel2}
Belyaev~A.~A.,
\emph{Characterization of spaces of multipliers for Bessel potential spaces},
Math. Notes, {\bf 96} (2014), № 5, 634--646

\bibitem{Bour}
Bourdaud~G.,
\emph{Localisations des espaces de Besov},
Stud. Math., {\bf 90} (1988), 153 -- 163

\bibitem{Fr}
 Franke~J.,
\emph{ On the spaces $F^s_{p, \: q}$ of Triebel-Lizorkin type: pointwise multipliers and spaces on domains},
 Math. Nachr., {\bf  125} (1986), 29 -- 68

%\bibitem{DjMi}
%Djakov~P., Mityagin~B.
%\emph{Criteria for existence of Riesz bases consisting of root functions of Hill and 1D Dirac operators},
%J. Funct. Anal., {\bf 263} (2012), \No 8, 2300 -- 2332


\bibitem{LRG}
Gala~S., Lemarie-Rieusset~P.~G.,
\emph{Multipliers between Sobolev spaces and fractional differentiation},
J. Math. Anal. Appl., {\bf 322} (2006), № 2, 1030 -- 1054


\bibitem{GerPaper}
Germain~P.,
\emph{Multipliers, para-multipliers and weak-strong uniqueness for the Navier-Stokes equation},
J. Differential Equations, {\bf 226} (2006), 373 -- 428

%\Bibitem{GerPaper} \by P.~Germain
%\paper Multipliers, para-multipliers and weak-strong uniqueness for the Navier-Stokes equation
%\jour J. Differential Equations
%\vol 226
%\yr 2006
%\pages 373 -- 428

%\bibitem{KMPaper}
%Kappeler~T., Mohr~C.
%\emph{Estimates for periodic and Dirichlet eigenvalues of the Schroedinger operator with singular potentials},
%J. Funct. Anal., {\bf 186} (2001), \No 1, 62 -- 91

%\Bibitem{KMohrPaper}  \by T.~Kappeler, C.~Mohr
%\paper Estimates for periodic and Dirichlet eigenvalues of the Schroedinger operator with singular potentials
%\jour J. Funct. Anal.
%\vol 186 : 1
%\yr 2001
%\pages 62 -- 91

\bibitem{KMS}
Кусаинова~Л.\,К., Мырзагалиева~А.\,Х., Султанаев~Я.\,Т.,
\emph{Об ограниченности оператора Шрёдингера в весовых пространствах Соболева},
Матем. заметки, {\bf 99} (2016), № 6, 945 -- 949

\bibitem{MSh}
Maz'ya~V.~G., Shaposhnikova~T.~O.,
\emph{Characterization of multipliers in pairs of Besov spaces},
Operator Theoretical Methods and Applications to Mathematical Physics: The Erhard Meister Memorial Volume, (2004), 365--387

\bibitem{MShbook}
Maz'ya~V.~G., Shaposhnikova~T.~O.,
\emph{Theory of Sobolev multipliers with applications to differential and integral operators},
Springer Verlag, Berlin – Heidelberg, 2009.

\bibitem{MV1}
Maz'ya~V.~G., Verbitsky~I.~E.,
\emph{The Schroedinger operator on the energy space: boundedness and compactness criteria},
Acta Math., {\bf 188} (2002), 263 -- 302

%\Bibitem{MV1} \by V.\,G.~Maz'ya, I.\,E.~Verbitsky
%\paper The Schroedinger operator on the energy space: boundedness and compactness criteria
%\jour  Acta Math.
%\vol 188
%\yr 2002
%\pages 263 -- 302

\bibitem{MV2}
Maz'ya~V.~G., Verbitsky~I.~E.,
\emph{The form boundedness criterion for the relativistic Schroedinger operator},
Ann. Inst. Fourier, {\bf 54} (2004), № 2, 317 -- 339

%\Bibitem{MV2} \by V.\,G.~Maz'ya, I.\,E.~Verbitsky
%\paper The form boundedness criterion for the relativistic Schroedinger operator
%\jour Ann. Inst. Fourier
%\vol 54 : 2
%\yr 2004
%\pages 317 -- 339

\bibitem{McLean}
McLean W.,
\emph{Strongly elliptic systems and boundary integral equations},
Cambridge University Press, Cambridge, 2000

\bibitem{NZSh1}
Нейман-Заде М.~И., Шкаликов А.~А.,
\emph{Операторы Шрёдингера с сингулярными потенциалами из пространств мультипликаторов},
Матем. заметки, {\bf 66} (1999), № 5, 723--733

%%% \Bibitem{NZSh1} Neiman-Zade M.I., Shkalikov A.A. Math. Notes 1999
%%М.\,И.~Нейман-Заде, А.\,А.~Шкаликов, ``Операторы Шрёдингера с сингулярными потенциалами из пространств мультипликаторов'', {\itshape Матем. заметки,} %%{\bfseries 66} : 5 (1999), 723 -- 733}}

\bibitem{NZSh2}
Neiman-Zade M.~I., Shkalikov A.~A.,
\emph{Strongly elliptic operators with singular coefficients},
Russian Journal Of Mathematical Physics, {\bf 13} (2006), № 1, 70--78

%%%\Bibitem{NZSh2} Neiman-Zade M.I., Shkalikov A.A. Strongly Elliptic Operators With Singular Coefficients // Russian Journal Of Mathematical Physics, 2006, \textbf{13} (1), pp. 70 -- 78



%\bibitem{SavSh}
%Савчук~А.\,М., Шкаликов~А.\,А.
%\emph{О собственных значениях оператора Штурма-Лиувилля с потенциалами из пространств Соболева},
%Матем. заметки, {\bf 80} (2006), \No 6, 864 -- 884

\bibitem{RSbook}
Runst~T., Sickel~ W.,
\emph{Sobolev spaces of fractional order, Nemytskij operators and nonlinear partial differential equations},
De Gruyter, Berlin, 1996

\bibitem{Si}
Sickel~W.,
\emph{On pointwise multipliers for $F^s_{p, q}(\mathbb{R}^n)$ in case $\sigma_{p, q} < s <\frac{n}{p}$},
 Ann. Mat. Pura Appl., IV, {\bf  176} (1999), 209 -- 250


\bibitem{SS}
Sickel W., Smirnov I.,
\emph{Localization properties of Besov spaces and of its associated multiplier spaces},
{\itshape Jenaer Schriften Math. Inf.,} Jena, 1999


\bibitem{SiTr}
Sickel~W., Triebel~H.,
\emph{Hoelder inequalities and sharp embeddings in function spaces of $B^s_{p, q}$ and $F^s_{p, q}$ type},
Z. Anal. Anwend., {\bf 14} (1995), № 1, 105--140


\bibitem{Str1}
Strichartz R.S.,
\emph{Multipliers on fractional Sobolev spaces},
J.  Math. Mech., {\bf 16} (1967),  1031 -- 1060

\bibitem{Tr}
Трибель~Х.,
\emph{Теория интерполяции, функциональные пространства, дифференциальные операторы},
Мир, М., 1980.

%%%%\Bibitem{Tr} Трибель Х. Теория интерполяции, функциональные пространства, дифференциальные операторы. М.: Мир, 1980\\

%%%\Bibitem{Tr} Triebel H. Interpolation theory, function spaces, differential operators. Amsterdam: North-Holland, 1978

%\bibitem{TrBook2}
%Triebel~H., NO REFERENCES CURRENTLY
%\emph{Function spaces II},
%Birkhauser, Basel, 1992

%%%\Bibitem{TrBook2} Triebel H. Function spaces OR Function spaces II


\end{thebibliography}
\end{document}